\documentclass[11pt]{article}

\usepackage{fullpage}
\usepackage{amsthm}
\usepackage{amssymb}
\usepackage{amsmath}
\usepackage{graphicx}
\usepackage{tikz}
\usepackage{mathrsfs}
\usepackage{float}
\usepackage{makecell}
\usepackage{amscd}
\usepackage{multirow}
\usepackage{cite}
\usepackage{verbatim}
\usepackage{hhline}
\usepackage{comment}

\newcommand{\E}{\mathbb{E}}
\newcommand{\Z}{\mathbb{Z}}

\newcommand{\diam}{\text{diam}}
\renewcommand{\epsilon}{\varepsilon}

\DeclareMathOperator{\lk}{lk}

\usepackage{relsize}
\newtheorem{theorem}{Theorem}
\newtheorem{lemma}[theorem]{Lemma}

\theoremstyle{definition}

\title{Complexes of nearly maximum diameter}
\author{Tom Bohman\thanks{Carnegie Mellon University. This work was supported by a grant from the Simons
Foundation (587088, TB)} \
 and Andrew Newman\thanks{Carnegie Mellon University}}
\begin{document}
\maketitle
\begin{abstract}
    The diameter of a strongly connected $d$-dimensional simplicial complex is the diameter of its dual graph. We provide a probabilistic proof of the existence of $d$-dimensional simplicial complexes with diameter $ (\frac{1}{d \cdot d!} - (\log n)^{-\epsilon}) n^d$. Up to the first order term, this is the best possible lower bound for the maximum diameter of a $d$-complex on $n$ vertices as a simple volume argument shows that the diameter of a $d$-dimensional simplicial complex is at most $ \frac{1}{d} \binom{n}{d}$. We also find the right first-order asymptotics for the maximum diameter of a $d$-pseudomanifold on $n$ vertices. 
    
\end{abstract}

\section{Introduction}
The {\em dual graph} of
a $d$-dimensional simplicial complex $X$ is the graph $\mathcal{G}(X) = (V, E)$ where $V$ is the set of $d$-dimensional faces of $X$ and $\{\sigma_1, \sigma_2\}$ is an edge if and only if $|\sigma_1 \cap \sigma_2| = d$, i.e. $\sigma_1$ and $\sigma_2$ meet at a face of codimension 1. If $\mathcal{G}(X)$ is connected we say that $X$ is \emph{strongly connected}. In the case that $X$ is strongly connected the (combinatorial) diameter of $X$ is the graph diameter of $\mathcal{G}(X)$\footnote{It is often the case in papers dealing with diameter questions for complexes that one assumes the complexes are \emph{pure dimensional}. But if $X$ is a $d$-complex $\mathcal{G}(X)$ can only see the pure $d$-part of $X$, so we won't concern ourselves with requiring that that $X$ is pure $d$-dimensional. }. Here we are interested in the question of the maximum possible diameter of a strongly connected $d$-dimensional complex on $n$ vertices. 

Questions about the diameter of arbitrary simplicial complexes was originally motivated by the Hirsch conjecture regarding the diameter of polytopes. The Hirsch conjecture asserted that if $X$ is a $d$-dimensional polytope with $n$ vertices then the diameter of $X$ is at most $n - d$. This conjecture was disproved by Santos in 2010 \cite{SantoHirschConj}, but a weaker version of the conjecture, the polynomial Hirsch conjecture, is still open. The polynomial Hirsch conjecture states that the diameter of a $d$-dimensional polytope with $n$ vertices is at most $p(n, d)$ where $p$ is a polynomial in $n$ and $d$.

Santos surveyed partial results related to the polynomial Hirsch conjecture and provided the first nontrivial lower bound on the diameter of an arbitrary simplicial complex in \cite{SantosDiamterSurvey}. Following Santos' notation we let $H_s(n, d)$ denote the maximum diameter of a $d$-dimensional simplicial complex on $n$ vertices. (Note though that we keep the $d$ parameter as the dimension of the complex, some other papers take $d$ to be the dimension of the complex plus one.) Santos proved
\[ H_s(n,d) \ge C\left(\frac{n}{d+1}\right)^{2(d+1)/3}  \]
for some absolute constant $C$ \cite[Corollary 2.12]{SantosDiamterSurvey}. In particular, 
once the view is broadened from simplicial polytopes to arbitrary simplicial complexes, a polynomial upper bound on the diameter cannot hold.

Santos also noted (see \cite[Corollary 2.7]{SantosDiamterSurvey}) that a simple volume argument gives an upper bound on the diameter of an arbitrary simplicial complex; namely, 
\[H_s(n, d) \leq \frac{n^d}{d \cdot d!}.\]
To prove this upper bound, suppose that we follow an induced path of length $\ell$ in the dual graph of a $d$-dimensional simplicial complex $X$. At the beginning of the path we see a $d$-simplex which has $d+1$ many $(d - 1)$-faces, from there each step of the path reveals a new $d$-simplex and $d$ new $(d - 1)$-faces. Thus after travelling along the path we have found $d+1 + d\ell$ many $(d - 1)$-dimensional faces. However the number of $(d -1)$ faces that a simplicial complex on $n$ vertices can have is at most $\binom{n}{d}$, so 
\[\ell \leq \frac{1}{d}\left(\binom{n}{d} - (d + 1)\right) \leq \frac{n^d}{d \cdot d!}.\]

Since the initial lower bound of Santos, progress has been made to improve the lower bound, but this simple volume upper bound is still the best upper bound that we have. Criado and Santos \cite{CriadoSantos}  showed that for each $d$ there are infinitely many $n$ so that
\[H_s(n, d) \geq \frac{n^d}{(d + 3)^d} - 3.\]
In particular, this shows that for $d$ fixed $H_s(n, d)$ grows like $\Theta(n^d)$, however there is still a gap in terms of $d$. In particular the ratio between the upper bound and this lower bound is exponential in $d$.
The lower bound was further improved using probabilistic techniques by Criado and Newman \cite{CriadoNewman}. They established 
\[H_s(n, d) \geq \frac{n^d}{5e(d+1)^2(d + 1)!}.\]
for $n$ sufficiently large. This therefore improves the ratio between the upper bound and lower bound to a $\Theta(d^2)$ factor. 

Here we settle the question of the maximum value of $H_s(n, d)$ for fixed $d$ up to the first-order term. We prove a new lower bound which asymptotically 
matches the simple volume upper bound as our main theorem.
\begin{theorem}\label{maintheorem}
For $d \geq 2$, $\epsilon < 1/ d^2$ and $n$ large enough there exists a pure $d$-dimensional simplicial complex on $n$ vertices whose dual graph is a path of length at least $(\frac{1}{d \cdot d!} - (\log n)^{-\epsilon}) n^d$. Therefore 
\[ H_s(n,d)  \ge \left(\frac{1}{d \cdot d!} - (\log n)^{-\epsilon}\right) n^d.  \]
\end{theorem}

We also consider the maximum possible diameter of a $d$-dimensional pseudomanifold on $n$ vertices. Recall that a  $d$-complex $X$ is said to be a pseudomanifold provided that each $(d - 1)$-face of $X$ is contained in exactly two $d$-faces. We let $H_{pm}(n, d)$ denote the maximum diameter of a $d$-dimensional pseudomanifold on $n$ vertices. The best-known bounds until now appear in \cite{CriadoNewman}, and the question of estimating $H_{pm}(n, d)$ had previously been considered in \cite{CriadoSantos}.
\begin{theorem}\cite[Theorem 1.3]{CriadoNewman}
For $d \geq 2$ we have
\[\frac{(1 - o_n(1))n^{d}}{4e(d + 1)!(d + 1)^4} \leq H_{pm}(n, d) \leq \frac{6n^d}{(d + 2)!}.\]
\end{theorem}
\noindent
Here we make an improvement in both the upper bound and the lower bound to the following:
\begin{theorem} \label{pmdiametertheorem}
For $d \geq 2$, $\epsilon < 2/(d^3 + d^2 - 2)$ and $n$ large enough 
\[\left( \frac{2}{(d + 1)(d + 1)!} - (\log n)^{-\epsilon} \right)n^d \leq H_{pm}(n, d) \leq \frac{2n^d}{(d + 1)(d + 1)!}.\]
\end{theorem}

The proof of the lower bound in Theorem~\ref{pmdiametertheorem} is  similar to the proof of the lower bound for $H_s(n, d)$ given in Theorem~\ref{maintheorem}, which we outline in the next section. The proof of Theorem~\ref{maintheorem} follows in Section~3. The proof of Theorem~\ref{pmdiametertheorem} is given in Sections~4~and~5. Some final remarks and comments regarding directions for further research are given in Section~6.

\section{A randomized complex construction process}
Our proof of Theorem~\ref{maintheorem} uses a probabilistic construction that we analyze using the differential equations method for establishing dynamic concentration. A probabilistic approach to the problem of finding a lower bound on $H_s(n, d)$ was previously applied in \cite{CriadoNewman}. But the process and associated dynamic concentration inequalities developed here are substantially different than the approach of \cite{CriadoNewman}, which used the Lov\'asz Local Lemma to show that with positive probability a simplicial complex chosen at random from a carefully selected distribution can be used to construct the desired complex of high diameter.  Here we introduce a simple randomized algorithm for the construction of a simplicial complex of high diameter. We then proceed to show that key statistics of this randomly constructed simplicial complex are tightly concentrated around their expected trajectories through many steps of the construction process. These dynamic concentration inequalities allow us to conclude that the randomized algorithm will succeed in producing the high diameter complex with high probability. For an introduction to the differential equations method for establishing dynamic concentration see \cite{Warnkesurvey} and \cite{Wormaldnotes}. 

Let $K^d_n$ denote the complete $d$-dimensional simplicial complex on $n$ vertices. Our goal is to find a long induced path in the dual graph of $K^d_n$. The dual graph of $K^d_n$ is commonly known as the Johnson graph $J(n, d + 1)$. The problem of computing $H_s(n, d)$ is the same as finding the length of the longest induced path in $J(n, d + 1)$.

For integers $ i < j$ we define $ [i,j] = \{i,i+1, \dots, j \}$.
Borrowing a definition of \cite{CriadoNewman}, we define the \emph{$d$-dimensional straight corridor on $N$ vertices}, denoted $SC_{d}(N)$,

to be the $d$-dimensional simplicial complex on $[N]$ whose facets are $[1, d  + 1], [2, d + 2], ..., [N - d, N]$, i.e. the facets are sequences of $d + 1$ consecutive numbers in $[N]$. We also will allow for the natural extension to $N = \infty$.

To find a long induced path in the dual graph of $K^d_n$ we want to find a large $N$ so that there is a simplicial map
$\phi: SC_d(N) \rightarrow K^d_n$
so that no two $(d - 1)$-faces in $SC_d(N)$ map to the same $(d -1)$-face of $K^d_n$. In doing so the dual graph of $SC_d(N)$ will match the dual graph of its image in $K^d_n$. The inductive structure of $SC_d(N)$ allows us to construct this map one vertex at a time. First we map the first $d + 1$ vertices of $SC_d(N)$ to any set of $d + 1$ vertices in $K^d_n$. Now we look at the $(d-1)$-face that is $\phi([2, d + 1])$ and set $ \phi(d+2)$ by picking a vertex uniformly at random from all vertices $v$ other than $\phi(1), \phi(2), \dots, \phi(d+1)$. More generally, after the images of the first $k$ vertices of the simplicial map are set, we consider the terminal $(d-1)$-face in the current image of the straight corridor, $\phi([k - (d - 1), k])$ in $K^d_n$, and we let $X_k$ denote the set of vertices $v \in K^d_n$ so that no $(d-1)$-faces of the $d$-simplex $\phi([k - (d - 1), k]) \cup \{v\}$, other than $\phi([k - (d - 1), k]),$ are in the image of $\phi$. We choose $v$ uniformly at random from 
$$ X_k \setminus \left\{ \phi(k -2d+1), \phi(k-2d+2), \dots, \phi(k) \right\}$$ 
and set $\phi(k + 1) = v$. Imposing the condition
that $ \phi(k+1)$ does not equal the image of any of the previous $2d$ vertices in the straight corridor simplifies the proof. Our task is to bound from below how long this process is likely to continue until we reach a place where $|X_k| \le 2d$. 

The intuition is that at each step $d$-many $(d - 1)$-faces are `deleted' in the sense that they cannot be traversed in future steps of the process, and over many steps these deleted faces should be randomly distributed across $ K_n^d$. So at step $i$, the $(d - 1)$-faces that haven't yet been deleted should resemble a random Linial--Meshulam $(d - 1)$-complex with $\binom{n}{d} - di$ many $(d - 1)$-faces. Recall that the Linial--Meshulam random $(d-1)$-complex is the probability space $\Delta_{d - 1}(n, p)$ sampled by starting with the complete $(d - 2)$-complex on $[n]$ and including each $(d-1)$-dimensional face independently with probability $p$. For our purposes the reader can largely view this as the same model as an Erd\H{o}s--R\'enyi $d$ uniform random hypergraph.

If it is indeed the case that the  remaining faces resemble a Linial--Meshulam random complex then $X_k$ would be distributed as a binomial random variable with $n - 2d$ trials and success probability $$\left(1 -  di \binom{n}{d}^{-1} \right)^d.$$ Under such an assumption, $X_k$ would be much larger than $2d$ with high probability all the way up to $i = (1 - o_n(1))\frac{n^d}{dd!}$ giving us our lower bound. Our proof establishes dynamic concentration inequalities that make this argument rigorous. The error bounds in the dynamic concentration argument determine the second-order term in Theorem~\ref{maintheorem}. 

\section{Proof of Theorem \ref{maintheorem}}
We consider the process described above of mapping $SC_d(N)$ into $K_n^d$ one vertex at a time. We begin by defining two random variables relative to this process. 
Let $ C_i$ be the collection of $ (d-1)$-faces in $ K_n^d$ that appear in the image of $ SC_d(d+i)$. This collection of $(d-1)$-faces is {\em closed} in the sense
that the process cannot choose a vertex that causes the image of the straight corridor to contain one of them a second time.
Let $A$ a $(d - 2)$-dimensional subcomplex of $K_n^d$. We let $ v_A$ denote the number of vertices in $A$, $|A|$ denote the number of $(k-2)$-faces in $A$, and  
$Y_A(i)$ be those vertices $v$ that for all $(d-2)$-faces $\sigma$ in $A$, $\sigma \cup \{v\} \not\in C_i$.  (We use the notation $ Y_A(i)$ for both the set defined here and its cardinality throughout this work. The meaning should always be clear in context.)

We begin the process simply choosing a $d$-face uniformly at random and setting $\phi(1)$, \dots , $\phi(d + 1)$ equal to the vertices of this $d$-face. We go from step $i$ to $i + 1$ (with step 1 being the selection of the starting $d$-face) by choosing the image of $d + i +1$. After step $i$ the number of ways to extend the path is $X_{d+i} = | Y_A(i) \setminus \{ \phi(i-d+1), \phi(i-d+2), \dots, \phi(i+d) \} | $ where $A$ the boundary of the $(d - 1)$-face $\phi([i+1, d+i])$. However in order to keep track of $Y_A$ when $A$ is the boundary of a $(d - 1)$-simplex it will be necessary to keep track of $Y_A$ for a few other choices of $(d-2)$-dimensional subcomplexes $A$. 

There is a significant difficulty in applying the differential equations method to establish dynamic concentration of $Y_A(i)$, where $A$ is the boundary
of a $(d-1)$-face. 
In a standard application of the method we introduce a martingale that balances the one-step expected change of a given variable, conditioned on the
history of the process, with the deterministic change in the variable's
trajectory equation over the corresponding interval of time. The variable $ Y_A$ is 'local' in the sense that it is based
at the fixed $ (d-2)$-complex $A$. This implies that if the currently last $ (d-1)$-face in the image of our straight corridor is far
from $A$ then the one-step conditional expected change in $ Y_A$ is zero, and this will spoil the martingale condition that we need.
We overcome this difficulty by adding a wrinkle to the method. We divide the process into a fixed number of subsequences and
introduce martingales that record the change in $Y_A$ relative to the change in the trajectory equation in each of these subsequences.
The number of steps in the underlying process spanned by a single step in one of the subsequences is enough for the location of the
last $(d-1)$-face in the image of the straight corridor to be almost uniformly random relative to the previous step observed by
the given subsequence.  In other words, after a small number of steps the process `forgets' the recent past and we can treat
the location of the terminal $(d-1)$-face as being nearly uniformly random. We note in passing that this technique can
be viewed as a variations on techniques developed in the study of random triangle removal \cite{BohmanFriezeLubetzky} and 
the triangle-free process \cite{BohmanKeevashr3t, FizGriffithsMorrisr3t}.  In those works 
one extends in space within the discrete structure itself to get to a suitably random position, and here we extend in time to
find a suitably random position.

For each $j \in \{0, ..., 3d\}$ and $A$ a $(d - 2)$-subcomplex of $K_n^d$, let $W_{A, j}(i)$ denote the \emph{total} number of vertices removed 
from $Y_A$ up to step $i$ in rounds congruent to $j$ mod $3d + 1$. Note that we have
\begin{equation}
\label{eq:adding}
Y_A(i) = n - v_A - \sum_{j=0}^{3d} W_{A,j}(i). 
\end{equation}
Following the intuition that $ C_i$ should resemble 
a Linial--Meshulam random $(d - 1)$-complex, we anticipate that we should have
\[W_{A, j}(i) \approx \frac{n}{3d+1} \left(1 - \left(1 - \frac{dd!i}{n^d}\right)^{|A|} \right).\]

The task at hand is to prove the following key lemma. 
Let $ {\mathcal A}$ be the collection of all $ (d-2)$-complexes $A$ such that $ v_A \le 2d$ and $ |A| \le  d^2$.
We establish a trajectory for $ W_{A, j}(i)$ by scaling time as
$t = \frac{i}{n^d} $
and defining
$p = 1 - dd!t$.
Note that our intuition is that $C_i$ is roughly the same as a binomial collection of $(d-1)$-faces where each face is chosen with probability $1-p$.
For a given error function $e = e(t)$ we define the stopping time $T= T_e$ to be the first step $i$ of the process for which there exists a fixed $(d-2)$-complex $A \in {\mathcal A}$ and
$j \in \{0, ..., 3d\}$ such that $ W_{A,j}$ does {\bf not} fall in the
interval
\[ I_{A}(t) := \left[  \frac{1}{3d+1} \left(n \left(1 - p^{|A|} \right) - \frac{n^{3/4} e(t)}{2} \right) ,   \frac{1}{3d+1} \left( n  \left(1 - p^{|A|} \right) + \frac{n^{3/4} e(t)}{2} \right) \right]. \]
Finally, we say that an event $E_n$ occurs \emph{with overwhelming probability} to mean that $\Pr(E_n) \geq 1 - n^{-\omega(1)}$. Note that if an event happens with overwhelming probability we also have that polynomially many identically distributed events simultaneously occur with overwhelming probability as well.  

\begin{lemma}\label{KeyDiffEqLemma}
Fix $d \geq 2$ and let $ 0 < \epsilon < 1/ d^2$. We have
\[ T_e \ge  \left(\frac{1}{dd!} - (\log n)^{-\epsilon}\right)n^d \]

with overwhelming probability, where $e(t) =  \exp \left\{  (10d+11) (1 - dd!t)^{-d^2} \right\}$.
\end{lemma}

We now note that Theorem \ref{maintheorem} follows directly from Lemma \ref{KeyDiffEqLemma}.
\begin{proof}[Proof of Theorem \ref{maintheorem}]
Take $\epsilon < 1/ d^2 $. Consider the path-mapping process from $SC_d(\infty)$ into $K^d_n$. 
Note that the process continues so long as $ X_{d+i} = Y_A(i) > 2d$, where $A$ is the boundary
of the final $ (d-1)$-face in the image of the straight corridor. It follows from (\ref{eq:adding}) that if $i< T_e$
and $ n p^{d} > n^{3/4} e(t) $ then $ Y_A(i) >2d$ and the process does not terminate at step $i$.  
At step
\[ i = i_{\rm end} := \left(\frac{1}{dd!} - (\log n)^{-\epsilon}\right)n^d, \] 
we have
\begin{gather*}
t =  \frac{1}{dd!} - (\log n)^{-\epsilon},  \ \ \ \ p = dd! (\log n)^{-\epsilon}, \ \ \ \  np^d = \Omega \left(  \frac{n}{ \log n} \right), \ \ \text { and } \\ 
e(t) = \exp \left\{ O \left( (\log n)^{ \epsilon d^2} \right)  \right\} =  n^{o(1)}. 
\end{gather*}
Thus, we have  $ n^{3/4} e(t) < np^d $ at $ i_{\rm end}$, and it follows that the event
that the process terminates before step $ i_{\rm end} $ is contained in the event that $ T_e < i_{\rm end}$. So Lemma~\ref{KeyDiffEqLemma} implies
that our process produces the desired simplicial complex with overwhelming probability.

\end{proof}

\subsection{Proof of Lemma~\ref{KeyDiffEqLemma}}

Recall that $ {\mathcal A}$ is the collection of all $ (d-2)$-complexes $A$ such that $ v_A \le 2d$ and $ |A| \le d^2$.
For the proof of Theorem~\ref{maintheorem} from Lemma~\ref{KeyDiffEqLemma} we were only interested in $Y_A$ for $A$ the boundary of a $(d - 1)$-simplex, and moreover we only required a lower bound on $Y_A$. That lower bound in turn depended on the upper bound on $W_{A, j}$ for the particular case that $A$ is a $(d - 1)$-simplex boundary. However through out the path mapping process the changes to $W_{A,j}$ for a subcomplex $A$ depends in a complicated way on $ Y_{A'}$
for various choices of $ A'$.  The collection ${\mathcal A}$ is chosen so that bounds on $Y_A$ for all $A \in {\mathcal A}$ are sufficient to establish concentration of $ W_{A,j}$ for all $A \in {\mathcal A}$.

We begin by noting the following, which follows from (\ref{eq:adding}),
\begin{equation}
    \label{eq:stopping}
i < T_e \ \ \Rightarrow \ \  Y_A(i) = n p^{|A|} \pm \frac{ n^{3/4} e(t)}{2} \pm O(1) \ \     
\text{ for all } A \in {\mathcal A} .
\end{equation}
For each $ A \in {\mathcal A}$ and each $ j \in \{0, \dots, 3d\}$ and each side of the
interval (upper and lower) we bound the probability that there is
a step $i < i_{\rm end}$ such that $ T_e = i$ because the variable $ W_{A,j}$ leaves the interval $ I_{A}(t)$
at step $i$ on the specified side of the interval. We show that each of these events is overwhelmingly 
unlikely.  As the number of choices of
$A$,$j$ and the side is polynomial in $n$, Lemma~\ref{KeyDiffEqLemma} follows from an application of the union bound.

Fix  $ A \in {\mathcal A}$ and $ j \in \{0, \dots, 3d\}$ and consider the desired upper bound on
$W_{A,j}$. We do not include the argument for the lower bound on $ W_{A,j}$ as it is essentially identical. However, in order that we have instances of the computations for both sides of the target interval, in the proof for pseudomanifolds in Section~4 below we carefully check the lower 
bound calculations while indicating that the upper bound is similar. Consider the sequence of 
random variables

\begin{equation}
\label{eq:mart}
Z_{A, j}(\ell) := W_{A, j}((3d + 1)\ell + j) - \frac{n}{3d + 1}\left( 1 - p^{|A|} \right) - \frac{n^{3/4} e(t)}{2(3d + 1)} 
%+ \frac{d}{3d + 1}.
\end{equation}
We want to show that $ Z_{A,j}(\ell) < 0$ for all $ \ell \le i_{\rm end}/(3d+1)$ with overwhelming probability.  
Note that $ Z_{A,j}(0) = - \Omega( n^{3/4} )$. Thus, in order for this variable to violate the desired condition
it needs to achieve a substantial positive change over the course of the process.  We show that this is overwhelmingly
unlikely by showing that sequence of $Z_{A, j} (\ell) $ is a supermartingale and 
applying a concentration lemma for supermartingales from \cite{BohmanTriangleFree}. 

In order to take full advantage of (\ref{eq:stopping}) we stop the
process of time $T_e$. Formally speaking, this means that we set $Z_{A,j}(\ell) $ equal to the expression in (\ref{eq:mart}) if
$ j + (3d+1)\ell \le T_e$ and we set $ Z_{A,j}(\ell)$ equal to $ Z_{A,j}( \ell-1)$ if $j + (3d+1)\ell > T_e$. There are two advantages
to stopping the process in this way.  First, we may assume that (\ref{eq:stopping}) holds when we are proving that the sequence is a martingale. 
The second advantage is if there is a value of $\ell $ such that $j + \ell(3d+1) \le i_{\rm end}$ and $ Z_{A,j}(\ell) \ge 0$
then we will hit the stopping time at that point, the whole process will `stop', and we will have 
$ Z_{A,j}( \lfloor i_{\rm end} /(3d+1) \rfloor) \ge 0$. It follows that the event that $ T_e < i_{\rm end}$ due to some random variable $ W_{A,j}$ exiting its target interval because it is too large is contained in the event
\[ \bigvee_{A\in {\mathcal A},j} \left\{ Z_{A,j}( \lfloor i_{\rm end} /(3d+1) \rfloor) \ge 0 \right\}.   \]
So, it suffices to show that event $ Z_{A,j}( \lfloor i_{\rm end} /(3d+1) \rfloor) \ge 0 $ is overwhelmingly unlikely
for all $A \in {\mathcal A}$ and $ j \in \{0,1, \dots , 3d\}$.

The first step in the proof is to establish that these sequences of random variables are indeed supermartingales.
\begin{lemma}\label{supermartigaleCondition}
Let $ (\mathcal{F}_i)_{i \ge 0}$ be the filtration defined by process, $ A \in {\mathcal A}$ and $ j \in \{ 0, \dots, 3d\}$. 
%
%If the state of the random path-mapping process at step $i = (3d + 1)\ell + j$, if 
%\[np^{|A|} - n^{3/4}e(t) \leq Y_A(i) \leq np^{|A|} + n^{3/4} e(t)\]
%%for $e(t) = C_1 \exp(C_2(1 - dd!t)^{-(2d - 2)}$ for all $A$, $(d - 2)$-subcomplexes of $K_n^d$ with $|A| \leq 2d^2$, then 
%for all $B$, $(d - 2)$-subcomplexes of $K_n^d$ with $|B| \leq 2d^2$, then
If $ (\ell+1)(3d+1) + j \le i_{\rm end}$ then we have
\[\E \left[Z_{A, j}(\ell + 1) - Z_{A, j}(\ell) \mid \mathcal{F}_{j + \ell(3d+1)} \right] \leq 0.\]
\end{lemma}
\noindent
We prove Lemma~\ref{supermartigaleCondition} below. Assuming that this Lemma holds, we apply the following version of the Hoeffding inequality, which is 
Lemma~7 of \cite{BohmanTriangleFree}, to conclude that $Z_{A,j}( \lfloor i_{\rm end}/(3d+1) \rfloor) < 0$ with overwhelming probability for any fixed $A$ and $j$. 
We say that a sequence of random variables $X_1, X_2, ...$ is {\em $(\eta, N)$ bounded} if 
$-\eta \leq X_{i + 1} - X_i \leq N$.
\begin{lemma}\cite[Lemma 7]{BohmanTriangleFree}
\label{lem:HA}
Suppose $\eta \leq N/10$ and $a < m\eta$. If $0 \equiv X_0, X_1, ...$ is an $(\eta, N)$-supermartingale then 
\[\Pr(X_m \geq a) \leq \exp \left\{ -\frac{a^2}{3\eta m N} \right\}.\]
\end{lemma}
\noindent
We apply this lemma to $Z_{A, j}(1), Z_{A, j}(2),..., Z_{A,j}(m)$ where $m$ is the largest integer so that  $(3d + 1)m + j \leq  i_{\rm end}$.
%n^d \left( \frac{1}{dd!} - (\log n)^{-\epsilon} \right)$. 
Consider $\ell <m$ and set
$t = ((3d+1)\ell +j)/n^d$. Note that, as
$e(t) = \exp\{ (10d+11) (1 - dd!t)^{-d^2})$ and $t \leq \frac{1}{dd!} - (\log n)^{-\epsilon}$ where $ \epsilon < 1/d^2$ we have $e'(t), e''(t) = n^{o(1)}$. Applying this bound on the second derivative we have
\begin{eqnarray*}
Z_{A, j}(\ell + 1) - Z_{A, j}(\ell) &=& W_{A, j}((3d + 1)(\ell + 1) + j) - W_{A, j}((3d + 1)\ell + j)  \\
&&  \ \ - \frac{1}{3d + 1} \cdot \frac{3d+1}{n^d} \left(n|A|p^{|A| - 1}(dd!) + \frac{ n^{3/4}e'(t)}{2} \pm O(n^{-d +1)})\right).
\end{eqnarray*}
%Since $e(t) = C_1 \exp(C_2 (\log n)^{-2d + 2})$ and $t \leq \frac{1}{dd!} - (\log n)^{-\epsilon}$, we see that $e'(t) = o(n^{\delta})$ for any $\delta > 0$. 
Note that 
\[0 \leq W_{A, j}((3d + 1)(\ell + 1) + j) - W_{A, j}((3d + 1)\ell + j) \leq d+1\]
since $W_{A, j}(i)$ is non-decreasing and at most $d+1$ vertices are among the $(d - 1)$-faces that are deleted (i.e. added to $ C_i$) in each step of the process. On the other hand
\begin{eqnarray*}
\frac{1}{n^d} \left(n|A|p^{|A| - 1}(dd!) + \frac{n^{3/4}e'(t)}{2} \pm O(n^{-d +1)})\right).
\end{eqnarray*}
is nonnegative and at most $ O(n^{1-d})$.
%\[\frac{dd!(3d + 1)d^2}{n^{d - 1}} + O(n^{-d + 3/2 + \delta})\]
%for any $\delta > 0$. 
Thus for $n$ sufficiently large we can apply Lemma~\ref{lem:HA} with $N = d+1$ and $\eta = C/n^{d - 1}$ for $C$ a large constant:
\[\Pr(Z_{A, j}(m) - Z_{A,j}(0) \geq a) \leq \exp\left(-\frac{a^2 n^{d - 1}}{3C(d+1) m}\right).\]
Recalling that $Z_{A, j}(0) = -\Omega(n^{3/4})$, we have 
\[\Pr(Z_{A, j}(m) \geq 0) \leq \exp\left\{- \Omega\left( \frac{n^{3/2 + d - 1}}{m}\right) \right\} 
= \exp\left\{ -\Omega\left( n^{1/2} \right) \right\}.\]
So we have $Z_{A, j}(m) < 0$ with overwhelming probability. It remains to prove the supermartigale condition.

\begin{proof}[Proof of Lemma \ref{supermartigaleCondition}]
In the interest of clarity, we write the error function as $ e(t)n^\alpha/2$, rather than
$ \exp\{ (10d+11) ( 1 - tdd!)^{-d^2} \} n^{3/4}/2$, throughout
this argument.  This is done in an effort to highlight how these parameters are chosen. Indeed, the
function $ e(t)$ is chosen to grow just quickly enough to ensure that the supermartingale condition is
maintained.  The power $\alpha$ needs to be large enough for the concentration inequality applied
above while also being small enough to maintain the supermartingale condition.
We also note that there are a number of small additional error terms
that appear in this calculation. We absorb these with room to spare by replacing $ e(t)n^\alpha/2$ with $e(t)n^\alpha$ in the expected value estimate.

Recall that ${\mathcal A}$ is the set of all $(d - 2)$-subcomplexes $A$ of $K_n^d$ so 
that $ v_A \le 2d$ and  $|A| \leq d^2$. For convenience of notation at step $i = (3d + 1) \ell + j$ we let  $\psi(a) = \phi(i+a)$ for $ a = \dots, -d, -(d-1), \dots, i+ 3d +1$.  So $ \{ \psi(-d), \psi(-d+1), \dots, \psi(0) \}$  is the simplex we start from and $ \psi(1), \psi(2), \dots, \psi(3d+1)$ are the vertex choices made in the ensuing $3d+1$ steps of the process. The process assigns $\psi(1), \psi(2), ..., \psi(3d + 1)$ one at a time. 

We begin by estimating the probability that on the $(3d + 1)$st step we cover a fixed $(d - 1)$-simplex $\sigma$ that has not yet been added to the path (i.e. we compute the probability
$ \sigma \in C_{i + 3d +1} \setminus C_{i+3d}$ for  $\sigma \notin C_i$.) Note that this event
is equivalent to the event $ \psi(3d+1) \in \sigma$ and $ \sigma \subset \{ \psi(2d+1), \psi(2d+2), \dots, \psi(3d+1) \}$. We estimate this probability
%$ \sigma$ is covered at the final step of the next $3d + 1$ steps 
under the
assumption $ i < T_e$, producing an estimate by bounding the number of ways the sequence
$ \psi(1), \psi(2), \dots, \psi(3d+1) $ can be chosen in agreement with this event.  We start
with the end of the sequence.  We have $dd!$ ways to assign the elements of $ \sigma $  to the elements of the sequence $ \psi(2d+1), \psi(2d+1), \dots, \psi(3d+1)$ so that $ \psi(3d+1)$ is assigned an element of $\sigma$. 
From here we choose the image of the unmapped vertex $b$ from among $2d + 1, ..., 3d$. There are 
$ Y_A(i)$ ways to
choose the image of $b$, where is $A$ is the boundary of $\sigma$. 
(Note that it is possible that some of the vertices counted in $ Y_A(i)$ might not still be available because 
these vertices are chosen in
the previous $2d$ steps and because some $(d-1)$-faces are selected in the steps between step $i$ and step $i+3d+1$, but the number of such vertices is at most $ O(1)$ and we can easily absorb this in the change in the error term discussed above.)
We then proceed to bound the number of possible choices 
for $ \psi(x)$ for $ x= 2d, 2d-1, \dots, 1$.
%
%there are $(np^{|A|} \pm n^{3/4}e(t))$ ways to do this (as the number of ways to do this is $Y_(A)(i)$. Now we need %to map the middle vertices $x = 1, ..., 2d$. 
%

Observe that in $SC_d(\infty)$ every vertex link is $SC_{d - 1}(2d)$ and the number of choices we have to map a vertex $x$ in $2d, 2d-1, \dots, 1$ depends on the collection of $(d-2)$-faces of $\lk(x)$ (within the infinite path) have already been mapped. 
%For $ k = 1, 2, \dots, d+1$ let $ {\mathcal F}_k$ be the collection of $k$-elements subsets $\alpha$  of $ -d, -d+1, \dots, 3d$ that are $(k-1)$-faces of the straight corridor on these vertices. (So $\alpha$ is a set of $k$ vertices that spans at most $ d+1$ positions in the sequence.)  We need to maintain the condition
%that $ \phi(\sigma) \not\in C_i$ for all $ \sigma \in {\mathcal F}_{d}$. So the number of choices of the image of $x$ is the number of vertices of $y$ with the property that if $\tau$ is a $ (d-1)$-face of the straight corridor such that $ x \in \tau$ and $ \phi( \tau \setminus \{x\}) = \alpha$ then $ \alpha \cup \{y\} \notin C_i$.  
For $x = 2d, 2d-1, \dots, 3$ the only $(d-1)$-faces of the straight corridor that include $x$ and vertices whose images have already been determined consist of $x$ and the $(d-1)$-element subsets of $\{ x+1, x+2, \dots, x+d \}$. Thus, the number of choices for $x$ is once again $ Y_A(i)$ where $ A$ is the boundary of a $(d-1)$-simplex. The remaining two case, $x=2$ and $x=1$, are more complicated as in these cases there are $(d-1)$-faces that
include $x$ and elements whose images have already been determined and are less than $x$. Indeed, when $x=1$, then the number of possible choices for $y = \psi(x)$ is at most $Y_{A}(i)$
where $ A$ is the collection of $ (d-2)$-faces in the image of $ \psi(-d+1), \dots, \psi(0), \psi(2), \dots, \psi(d+1)$, which can viewed as the collection of

$(d - 2)$-faces in the $(d - 1)$-dimensional straight corridor on $[2d]$. 
%with vertices $d + 1, ..., 2d - j$ removed for each $j \in 1,..., d$. If $j = d$ then we are deleting the empty set and 
It is straightforward to see that the number of codimension $k$ faces in a $D$-dimensional straight corridor on $N$ vertices is
\[f_{D - k}(SC_D(N)) = \binom{D}{D - k + 1} + (N - D)\binom{D}{k}.\]
So, for $x = 1$ the number of $(d - 2)$-faces of $ \lk{(x)} $ that have already been mapped when we assign $ \psi(x)$ is $d^2$. The need for this estimate is the reason that we track $Y_A$ for $A$ with $v_A$ as large as $2d$ and $|A|$ as large as $d^2$. Now we turn to the final case, which is $x = 2$. Given the images $\psi(-d), \dots, \psi(0), \psi(3), \dots \psi(3d+1)$, the images of $(d - 2)$-faces of $ \lk{(x)}$ are naturally partitioned into those that intersect $ \{\psi(-d), \psi(-d+1), \dots, \psi(0)\} $ and
those that do not. There are exactly $d-1$ such faces in the first category as such a face is determined by its minimum element. The faces in the second category are simply the boundary of a $(d-1)$-face. So, for $x=2$ there are $2d-1$ many $(d-2)$-faces of $ \lk{(x)}$ whose images have
been set when we choose $ \psi(x)$.
%on the left (i.e. among the first $d$ vertices), those that cross the gap, and those on the right. Clearly the number of $(d - 2)$-faces on the left is always $d$. The number that cross is zero if $j \leq d - 2$ and is $d - 2$ in the case that $j = d - 1$. The number of vertices on the right is $2d - (2d - j) = j$. If $j = d - 1$, the sequence $2d - (d - 1) + 1 = d + 2, ..., 2d$ form a $(d - 2)$-face, so if $j = d - 1$ we have one $(d - 2)$-face on the right, if $j \leq d - 2$ then there aren't enough vertices on the right to form a $(d-2)$-face. Thus for $j = d$ we have $|A| = d^2$, for $j = d - 1$ we have $|A| = d + d - 2 + 1 = 2d - 1$ and for $j \leq d - 2$ we have $|A| = d$. 
Putting all of this together, and noting that every simplicial complex $A$ considered here is in the collection ${\mathcal A}$, we see that the number of choices for $ \phi(i+1), \dots, \phi(i+3d+1)$ so that 
$ \sigma \in C_{i +3d+1} \setminus C_{i+3d} $ is 
%at most
\[ dd!(np^d \pm n^{\alpha} e(t))^{2d - 1} (np^{2d - 1} \pm n^{\alpha}e(t))(np^{d^2} \pm n^{\alpha}e(t)).\]
It follows that we have
%Now we have that while $Y_A(i) = np^{|A|} \pm n^{\alpha} e(t)$, the probability that a portion of the %path of length $3d + 1$ conditioned on the starting position ends at a particular $(d - 1)$-face $\tau$ is
%\[ ,\]
%for all $ \sigma \notin C_i$.
%
%If $ i \le T_e$ then the probability that a particular $\tau \not\in C_i$ is covered in the last step of a %portion of the path of length $3d + 1$ is at most
\begin{equation*}
\begin{split}
{\mathbb P} \left( \sigma \in C_{i+3d+1} \setminus C_{i+3d} \mid i \right.  & \left. < T_e \right) \\
&= \frac{dd!(np^d \pm n^{\alpha} e(t))^{2d - 1} (np^{2d - 1} \pm n^{\alpha}e(t))(np^{d^2} \pm n^{\alpha}e(t))}{(np^{d} \pm n^{\alpha} e(t))^{3d + 1}} \\
& \leq dd!\frac{(np^d + n^{\alpha} e(t))^{2d - 1} (np^{2d - 1} + n^{\alpha}e(t))(np^{d^2} + n^{\alpha}e(t))}{(np^{d} - n^{\alpha} e(t))^{3d + 1}}\\
%&\leq& dd!\frac{n^{2d + 1}p^{3d^2 + d - 1} + (2d+1)n^{2d}p^{2d^2 + d - 1}n^{\alpha}e(t) + O(n^{2d-1}p^{2d^2 - d}n^{2\alpha}e(t)^2)}{n^{3d + 1}p^{3d^2 + d} - (3d + 1)n^{3d}p^{3d^2}n^{\alpha}e(t) - O(n^{3d - 1}p^{3d^2 - d}n^{2\alpha} e(t)^2)} \\
&\leq dd!\left(\frac{1}{n^d p} + \frac{(2d+2) e(t) }{n^{d+1 - \alpha}p^{d^2+1}} +
\frac{ (3d+2) e(t)}{ n^{d+1- \alpha} p^{d+1}} \right) \\
%+ \frac{2 n^{5d + 1 - \alpha} p^{5d^2 - 1} e(t) + (3d + 1) n^{5d + 1 - \alpha} p^{6d^2 + d - 1} e(t)}{n^{6d + 2}p^{6d^2 + 2d}}\right) \\
&\leq dd!\left(\frac{1}{n^d p} + \frac{(5d+4) e(t) }{n^{d+1 - \alpha}p^{d^2+1}} \right). 
%&\leq& dd!\left(\frac{1}{n^dp} + \frac{2e(t)}{n^{d + 1 - \alpha} p^{d^2 + 2d - 1}} + O\left( \frac{1}{n^{d + 1 - \alpha} p^{d + 1}}\right) \right)
%&\leq& \frac{1}{n^d p} + \frac{n^{3d + 1}p^{3d^2 + d}(n^{2d}p^{2d^2 - d - 1}n^{\alpha}e(t) + O(n^{2d}p^{3d^2 - d}n^{2\alpha}e(t)^2)) + n^{2d + 1}p^{3d^2+d - 1}((3d + 1)n^{3d}p^{3d^2}n^{\alpha}e(t) + O(n^{3d - 1}p^{3d^2 - d}n^{2\alpha} e(t)^2))}{n^{3d + 1}p^{3d^2 + d}(n^{3d + 1}p^{3d^2 + d} - (3d + 1)n^{3d}p^{3d^2}n^{\alpha}e(t) - O(n^{3d - 1}p^{3d^2 - d}n^{2\alpha} e(t)^2))}
\end{split}
\end{equation*}
for all $\sigma \notin C_i$.
%Now for each $A \in \mathcal{X}_d$ and $j = 0, ..., 3d$, we denote by $W_{A, j}(i)$ the number of vertices removed from $Y_A$ during the first $i$ steps in rounds congruent to $j \mod 3d + 1$. We have show that as long as $Y_A = np^{|A|} \pm n^{\alpha}e(t)$, the expected change to $W_{A, j}(i)$ for any $A$, $j$ is 
It follows that we
\[ {\mathbb E}[ W_{A,j}(i + 3d +1) - W_{A,j}(i) \mid {\mathcal F}_i] \le |A|Y_A(i) dd!\left(\frac{1}{n^d p} + \frac{(5d+4) e(t) }{n^{d+1 - \alpha}p^{d^2+1}} \right) \]
%\left(\frac{dd!}{n^dp} + \frac{2dd!e(t)}{n^{d + 1 - \alpha} p^{d^2 + 2d - 1}} + O\left( \frac{1}{n^{d + 1 - \alpha} p^{d + 1}}\right) \right)  \]
%\left(\frac{dd!}{n^dp} + \frac{2dd!e(t)}{n^{d + 1 - \alpha} p^{d^2 + 2d - 1}} + O\left( \frac{1}{n^{d + 1 - \alpha} p^{d + 1}}\right) \right)\]
%
%If $W_{A, j}(i)$ behaves as it does in a random graph the we expect
%\[W_{A, j}(i) = \frac{n}{3d + 1} \left(1 - p^{|A|}\right)\]
Recalling that we set
$Z_{A, j}(\ell) = W_{A, j}((3d + 1) \ell + j) - \frac{n(1 - p^{|A|}) + n^{\alpha}e(t)/2}{3d + 1}$, 
%then we can use what we just showed to conclude that unless $p$ is very small, with overwhelming probability $Z_{A, \ell}$ is always negative.  
we have 
\begin{eqnarray*}
\mathbb{E}(Z_{A, j}(\ell + 1) - Z_{A, j}(\ell) \mid \mathcal{F}_i) &\leq& |A|Y_A(i)  
 dd!\left(\frac{1}{n^d p} + \frac{(5d+4) e(t) }{n^{d+1 - \alpha}p^{d^2+1}} \right)\\
%\left(\frac{dd!}{n^dp} + \frac{2dd!e(t)}{n^{d + 1 - \alpha} p^{d^2 + 2d - 1}} + 
%
%O\left( \frac{1}{n^{d + 1 - \alpha} p^{d + 1}}\right) \right) \\
&&  \ \ - \frac{ 3d+1}{ n^d} \cdot \frac{n|A| p^{|A| - 1} dd! + n^{\alpha}e'(t)/2}{3d + 1} + O(n^{-(2d - 1)})  \\
 &\leq& |A|(np^{|A|} + n^{\alpha}e(t)) dd!  
  \left(\frac{1}{n^d p} + \frac{(5d+4) e(t) }{n^{d+1 - \alpha}p^{d^2+1}} \right)\\
&&  \ \ - \frac{n|A| p^{|A| - 1} dd! + n^{\alpha}e'(t)/2}{ n^d} + O(n^{-(2d - 1)})  \\
 %\left(\frac{1}{n^dp} + \frac{2dd!e(t)}{n^{d + 1 - \alpha} p^{d^2 + 2d - 1}} + O\left( \frac{1}{n^{d + 1 - \alpha} p^{d + 1}}\right) \right) -\\
%&& \frac{n|A| p^{|A| - 1} dd!(3d + 1) + n^{\alpha}e'(t)}{n^d(3d + 1)} + O(n^{-(2d - 1)})  \\
&\leq& \frac{dd!|A| e(t)}{n^{d - \alpha} p} + \frac{(5d+4)|A| dd! e(t)}{n^{d - \alpha} p^{d^2 + 1 - |A|}} - \frac{e'(t)}{ 2n^{d -\alpha}} + O\left(  \frac{ e(t)^2}{n^{d + 1 - 2\alpha }p^{d^2 + 1}} \right) \\
&\leq& \frac{(5d+5) d^3 d! e(t)}{n^{d - \alpha} p^{d^2}} - \frac{e'(t)}{ 2n^{d -\alpha}} + O\left(  \frac{ e(t)^2}{n^{d + 1 - 2\alpha }p^{d^2 + 1}} \right).
%&\leq& |A|(np^{|A|} + n^{\alpha} e(t))  \left(\frac{1}{n^dp} + \frac{2e(t)}{n^{d + 1 - \alpha} p^{d^2 + 2d - 1}} + O\left( \frac{1}{n^{d + 1 - \alpha} p^{d + 1}}\right) \right) -\\
%&& n^{1 - d}|A| p^{|A| - 1} dd! - \frac{n^{\alpha - d}e'(t)}{3d + 1} + O(n^{-(d + 1)}) 
\end{eqnarray*}
%So each complex in $\mathcal{X}_d$ has $|A|$ between $1$ and $2d^2$. So in order for the third term in the bound above to be the dominant term, and hence 
To make this expected difference negative we choose $e(t)$ so that
\[e'(t) \geq \frac{(10d+11)d^3d!e(t)}{p^{d^2}} =  \frac{(10d+11)d^3d!e(t)}{ (1 - dd!t)^{d^2}}.\]
So it suffices to take 
$$e(t) = \exp\left\{ (10d+11) (1 - dd!t)^{-d^2} \right\}. $$ 
Note that so long as $t \leq \frac{1}{dd!} - \omega((\log (n))^{-1/d^2})$ we have $ e(t) = n^{o(1)}$.
%
%\commAN{The constant in the above lower bound should be checked again as the definition of $\mathcal{X}_d$ has been changed}
%
%
%
%For this choice of $e(t)$ our estimate on $Y_A(i)$ is a nontrivial estimate until
%\[n^{\alpha} \exp(C_2(1 - dd!t)^{-(2d - 2)}) \geq n(1 - dd!t)^{d^2}\]
%So our estimates work as long as $t \leq \frac{1}{dd!} - \omega((\log (n))^{-1/(2d - 2)})$.

\end{proof}

\section{Lower bound for pseudomanifolds}
For the lower bound of Theorem \ref{pmdiametertheorem}, we use a modified path mapping process. We observe that for each $N$ and $d$, the boundary of $SC_{d + 1}(N)$, denoted $\partial SC_{d+1}(N)$ is a $d$-sphere. If we can map $\partial SC_{d + 1}(N)$ into $K_n^d$ so that no $(d - 1)$-faces are ever repeated then its image is a $d$-complex on $n$ vertices with diameter equal to the diameter of $\partial SC_{d + 1}(N)$, and more over its image is a pseudomanifold since the map is injective on $(d -1)$-faces and on $d$-faces. Mapping $\partial SC_{d + 1}(N)$ into $K_n^d$ is also the idea of the pseudomanifold case of \cite{CriadoNewman}, the difference though is that we find our map using the differential equations method for dynamic concentration. A sufficient lower bound on the diameter of $\partial SC_{d + 1}(N)$ is already known.
\begin{lemma}\cite[Lemma 3.2]{CriadoNewman}\label{boundarydiameter}
The diameter of $\partial SC_{d + 1}(N)$ is at least $\frac{d}{d + 1} N - d - 1.$
\end{lemma}

We define the $d$-pseudomanifold mapping process to be the random process that maps $\partial SC_{d + 1}(\infty)$ into $K_n^d$ one vertex at a time subject to the rule that no $(d - 1)$-face is repeated. As in the previous case we let $C_i$ be the collection of $ (d-1)$-faces that appear in the image of the straight corridor after $i$ steps of the process. Furthermore, we again let $Y_A(i)$ for $A$ a $(d - 2)$-subcomplex be the number of vertices $v$ so that for all $\sigma \in A$, $\sigma \cup \{v\}$ is not in $C_i$. We will again
apply the intuition that these variables resemble their counterparts in a Linial--Meshulam random $(d - 1)$-complex on $n$ vertices.
%, for a particular set $A$ of $(d - 2)$-faces, the expected value of $Y_A$ is $np^{|A|}$, 
We will show that these approximations are sufficiently close as long as $i$ is far enough below $\binom{d + 1}{2}d! n^d$. The precise stopping time is given later. 

As above, we let $W_{A, j}(i)$ for $j \in \{0, ..., 3(d + 1)\}$ and $A$ a bounded $(d - 2)$-subcomplex of $K_n^d$ be the number of vertices that are removed from $Y_A$ by step $i$ in rounds congruent to $j$ mod $3(d + 1) + 1$. For the pseudomanifold case we again establish trajectories for $W_{A, j}(i)$ for all $A$ in some reasonable collection $\mathcal{A}$. Here it suffices to take $\mathcal{A}$ to be all $(d - 2)$-subcomplexes $A$ of $K_n^d$ with $ v_A \le 2(d+1)$ and $|A| \leq d + (d + 2) \binom{d}{2}$. We take
\[t := i/n^d \ \ \ \ \text{ and } \ \ \ \ 
p := 1 - \binom{d + 1}{2} d! t.\]

The choice of $\binom{d + 1}{2}d!$ in the definition of $p$ comes from the volume estimate in the mapping process. Every time we map a new vertex we uncover $\binom{d + 1}{2}$ new $(d - 1)$-faces since we add all $(d - 1)$-faces of a $(d + 1)$-simplex that happen to contain some fixed vertex of that simplex. So trivially the number of steps in the process cannot exceed:
\[\binom{n}{d}/\binom{d + 1}{2}.\]

For an error bound $e(t)$, we let $T = T_e$ be the first step $i$ of the random pseudomanifold mapping process where there is some $A$ with $v_A \le 2(d+1)$ and $|A| \leq d + (d + 2)\binom{d}{2}$ and $j \in \{0, ..., 3(d + 1)\}$ so that $W_{A, j}$ does not fall in the interval
\[ \left[  \frac{1}{3d+4} \left(n \left(1 - p^{|A|} \right) - \frac{n^{3/4} e(t)}{2} \right) ,   \frac{1}{3d+4} \left( n  \left(1 - p^{|A|} \right) + \frac{n^{3/4} e(t)}{2} \right) \right]. \]

Our stopping time lemma is the following:
\begin{lemma}\label{pmStoppingTime}
Fix $d \geq 2$ and let $\epsilon < 2/(d^3 + d^2 - 2)$. We have
\[T_e \geq \left(\frac{1}{d!\binom{d + 1}{2}} - (\log n)^{-\epsilon} \right)n^d\]
with overwhelming probability where \[e(t) = \exp \left\{16d \left(1 - d!\binom{d + 1}{2}t\right)^{-(d^3/2 + d^2/2 - 1)}\right\}.\]
\end{lemma}

Lemmas \ref{boundarydiameter} and \ref{pmStoppingTime} are sufficient to prove the lower bound in Theorem \ref{pmdiametertheorem}.
\begin{proof}[Proof of lower bound in Theorem \ref{pmdiametertheorem}]
By Lemma \ref{pmStoppingTime} for any $\epsilon < 2/(d^3 + d^2 - 2)$, one can find a diameter-preserving quotient of 
\[\partial SC_{d + 1}\left(\left(\frac{1}{d!\binom{d + 1}{2}} - (\log n)^{-\epsilon}\right)n^d\right)\]
in $K_n^d$.
By Lemma \ref{boundarydiameter} such a subcomplex of $K_n^d$ has diameter at least 
\[\frac{d}{d + 1} \left(\frac{1}{d!\binom{d + 1}{2}} - (\log n)^{-\epsilon}\right)n^d - d - 1.\]
The lower bound of Theorem \ref{pmdiametertheorem} follows. 
\end{proof}

For fixed $\epsilon < 2/(d^3 + d^2 - 2)$, we let $i_{\text{end}} = \left(\frac{1}{d!\binom{d + 1}{2}} - (\log n)^{-\epsilon} \right)n^d$.
Let 
\[Z_{A, j}(\ell) := W_{A, j}((3d + 4)\ell + j) - \frac{1}{3d+4} \left(n \left(1 - p^{|A|} \right) - \frac{n^{3/4} e(t)}{2} \right).\]
We show that with overwhelming probability $Z_{A, j}(\ell) > 0$ for $i = (3d + 4)\ell + j$ at most $i_{\text{end}}$

For the case of simplicial complexes we carefully verified that the upper bound on the $W_{A, j}$'s hold with overwhelming probability, and then said that the lower bound argument is similar; here we do the opposite so that the reader sees careful proofs of both an upper bound and a lower bound.

\begin{lemma}\label{submartingaleCondition}
Let $(\mathcal{F}_i)_{i \geq 0}$ be the filtration defined by the pseudomanifold path-mapping process, $A \in \mathcal{A}$ and $j \in \{0,..., 3d + 3\}$. If $(\ell + 1)(3d + 4) + j \leq i_{\text{end}}$ then we have
\[\E(Z_{A, j}(\ell + 1) - Z_{A, j}(\ell) \mid \mathcal{F}_{j + \ell(3d + 4)}) \geq 0.\]
\end{lemma}
\begin{proof}
Suppose we are at step $i = (3d + 4)\ell + j$ and let $\mathcal{F}_i$ be the state of the path so far. We are thus at the boundary of a $(d + 1)$-face on vertices denoted $[\psi(-d - 1), ..., \psi(0)]$. At each step we choose $\psi(x)$ to be a vertex of $Y_A$ for $A$ the codimension 2 skeleton of $[\psi((x - 1) - d), ..., \psi(x - 1)]$. For $ i < T_e$
%Y_A = np^{|A|} \pm n^{\alpha}e(t)$ for all $A$ with $|A| \leq d + (d + 2) \binom{d}{2}$ 
we bound from below the probability that $\sigma$ an unassigned $(d - 1)$-face is covered at step $(3d + 4)(\ell + 1)$ conditioned on a starting point at step $(3d + 4) \ell$. 

Let $[u_1, \dots, u_d]$ be a fixed face of $K_n^d$ that has not been assigned yet at step $(3d + 4)\ell$. We know $[\psi(- d  -1), \dots, \psi(0)]$ and we map vertices $1$, $2$, \dots, $3d + 4$ one at a time. At each step we add a cone over the codimension-2 skeleton of a $d$-simplex. In order for $[u_1, \dots, u_d]$ to be covered when we map vertex $3d + 4$ it must be the case that vertex $3d + 4$ and $d - 1$ of the vertices $\{2d + 3, ..., 3d + 3\}$ are mapped to $[u_1, ..., u_{d}]$. So we have $d$ choices for mapping $3d + 4$ and then have $\binom{d + 1}{2}$ choices to pick $d - 1$ vertices from among $\{2d + 3, ..., 3d + 3\}$ to map to the remaining $u_i$'s. Once these vertices are selected we have $(d - 1)!$ ways to map them. 

So we have $d!\binom{d + 1}{2}$ ways to map $\{2d + 3, \dots, 3d + 4\}$ so that $[u_1, \dots, u_d]$ is among the final $(d - 1)$-faces added. At that point we have two unmapped vertices among $\{2d + 3, \dots, 3d + 4\}$ to map. The image of the $d$-many mapped vertices among $\{2d + 3, \dots, 3d + 4\}$ at the point form a $(d - 1)$-simplex. We have at least $np^{d} - n^{\alpha}e(t)$ ways to pick the image of the first unmapped vertex from among $\{2d + 3, ..., 3d + 4\}$. At that point we have at least $np^{\binom{d+1}{2}} - n^{\alpha} e(t)$ ways to pick the second unmapped vertex from among $\{2d + 3, ..., 3d + 4\}$.  Now we map the vertices $1, ..., 2d + 2$ one at a time. Within $\partial SC_{d + 1}(\infty)$ the link of a vertex is $\partial SC_d(2(d + 1))$. When we go to map a vertex $x$ we will always have the first $d + 1$ vertices of the link are already assigned and toward the end some of the final $d + 1$ vertices of the link are assigned too. 

Similar to the simplicial complex case, we fix $x \in \{1, \dots, 2d + 2\}$ and we denote the vertices of $\lk(x)$, by $1, \dots, (d + 1), d + 2, \dots, 2d + 2$. The assigned vertices will always be $1, \dots, d + 1$, and some tail of $d + 2, \dots, 2d + 2$. Our task is to count $(d - 2)$-faces of the subcomplex of $\partial SC_{d}(2d + 2)$ obtained by removing vertices $d + 2$, \dots, $d + 2 + j$ for each $-1 \leq j \leq d$. In all cases we have $\binom{d + 1}{2}$ $(d - 2)$-faces from among the first $d + 1$ vertices.  In order to have even a single $(d - 2)$ face survive from the last $d + 1$ vertices we must have $d - j \geq d - 1$, so $j \leq 1$ must hold in order to be in one of the exceptional cases. Recall that there were two exceptional cases for simplicial complexes, here there will be three. If $j \geq 2$ then the number of $(d - 2)$-faces that survive is 
\[\binom{d + 1}{d - 1} = \binom{d + 1}{2}.\]

If $j = 1$ then we have one $(d - 2)$-face among $d + 4$, \dots, $2d + 2$. A $(d - 2)$-face of $SC_{d}(2d + 2)$ (or equivalently of its boundary) is a subset of size $d - 1$ from a sequence of consecutive $d + 1$ vertices. From this a ``crossing $(d - 2)$-face", that is a face with some vertices in $[1, d + 1]$ and some vertices in $[d + 2, 2d + 2]$, can be made by choosing $4 \leq i \leq d + 1$ as the starting vertex and taking $\{i, ..., i + d\} \setminus \{d + 2, d + 3\}$ as a $(d - 2)$-face. Thus if $j = 1$ there are $d - 2$ crossing $(d - 2)$-faces. So for $j = 1$ the number of $(d - 2)$ faces is:
\[\binom{d + 1}{2} + d - 2 + 1.\]

If $j = 0$ then there are $d$ choices for $(d - 2)$-faces from among $d + 3, \dots, 2d + 2$. For the crossing $(d-2)$-faces we observe that a crossing face can be uniquely identified by its starting vertex $i \leq d + 1$ so that $i + d \geq d + 3$ and its deleted vertex in $[i, i + d]$ different from $i$ and from $d + 3$. There are $d - 1$ choices for $i$, and then $d - 1$ choices for the other vertex to remove except that when $d + 2 = i + d - 1$ there is only one choice to build a crossing $(d-2)$-faces since we cannot pick the unique vertex on the other side of the gap to delete. Thus the number of crossing $(d - 2)$-faces is $(d - 1)^2 - 1$. So if $j = 0$ the number of $(d - 2)$-faces is
\[\binom{d + 1}{2} + d + (d - 1)^2 - 1.\]

 If $j = -1$, then the number of $d - 2$ faces is simply the number of codimension 2 faces in $SC_{d}(2d + 2)$, which we already know to be 
\[d + (d + 2)\binom{d}{2}.\]

 So putting this all together we have 3 special cases for $j$, and we obtain that the number of ways to choose images for $\{1, ..., 2d + 2\}$ is at least
\[\left(np^{\binom{d + 1}{2}} - n^{\alpha}e(t)\right)^{2d - 1}
\left(np^{d +(d + 2)\binom{d}{2}} - n^{\alpha} e(t) \right) 
\left( np^{\binom{d + 1}{2} + d +(d - 1)^2 - 1 } - n^{\alpha} e(t) \right)
\left( np^{\binom{d + 1}{2} + d - 1} - n^{\alpha} e(t) \right) 
.\]
Multiplying this by the 
\[\binom{d + 1}{2}d! (np^d - n^{\alpha}e(t))(np^{\binom{d + 1}{2}} - n^{\alpha}e(t))\]
for the lower bound on the number of ways to map $\{2d + 3, ..., 3d + 3\}$, and ignoring the error terms temporarily, we get that the probability that $[u_1, ..., u_d]$ is covered in the last step is roughly:
\[\binom{d + 1}{2} d! \frac{n^{2d + 4} p^{3d^3/2 + 7d^2/2 + 2d - 1}}{(np^{\binom{d + 1}{2}})^{3d + 4}} = d! \binom{d + 1}{2} \frac{1}{n^d p}.\]
This matches with our intuition that a comparison with the Linial--Meshulam model should hold.

%
%This is what we would expect if the included faces were independent, so the analogy to random complexes we wish to make precise still holds up as far as the rough calculation. 
Making this intuition precise though and keeping the error terms, we have the the probability that $[u_1, ..., u_d]$ is covered at step $3d + 4$ for $n$ sufficiently large is at least.
\[\binom{d + 1}{2}d! \frac{n^{2d + 4}p^{3d^3/2 + 7d^2/2 + 2d - 1} - (2d + 4) n^{2d + 3 + \alpha}p^{d^3 + 3d^2 + 2d - 1}e(t)}{n^{3d + 4}p^{\binom{d+1}{2}(3d + 4)} + (3d + 5)n^{3d + 3 + \alpha}p^{\binom{d + 1}{2}(3d + 3)}e(t)}.\]
For $n$ sufficiently large this is at least
%\[
%\binom{d + 1}{2} d! \left(\frac{1}{n^dp} - \frac{(2d + 4)n^{5d + 7 + \alpha}p^{5d^3/2 + 13d^2/2 + 4d - 1}e(t) - (3d + 5)n^{5d + 7 + \alpha}p^{3d^3 + 13d^2/2 + 7d/2 - 1}e(t)}{n^{6d + 8}p^{\binom{d +1}{2}(6d + 8)}} \right) 
%\]
\begin{multline*}
\binom{d + 1}{2} d! \left(\frac{1}{n^dp} - \frac{(2d + 5)n^{5d + 7 + \alpha}p^{5d^3/2 + 13d^2/2 + 4d - 1}e(t)}{n^{6d + 8}p^{\binom{d +1}{2}(6d + 8)}} \right) \\ \geq \binom{d+1}{2}d!\left(\frac{1}{n^dp} - \frac{(2d + 5)e(t)}{n^{d + 1  - \alpha}p^{d^3/2 + d^2/2 + 1}} \right).
\end{multline*}

As in the case for simplicial complexes, it follows that for $i < i_{\text{end}}$, 
\[ {\mathbb E}[ W_{A,j}(i + 3d +4) - W_{A,j}(i) \mid {\mathcal F}_i] \ge |A| Y_A(i)  \binom{d + 1}{2} d! \left( \frac{1}{n^dp} - \frac{(2d + 5) e(t)}{n^{d + 1 - \alpha} p^{d^3/2 + d^2/2 + 1}} \right). \]
We note in passing that there are a number of small error terms that are easily absorbed by replacing error terms of the form $ e(t)n^\alpha/2$ with $ e(t) n^\alpha$. These include the small changes in the collection of available $ (d-1)$-faces during the $3d+4$ steps of the process we analyze here and the vertices that are excluded because they appeared among the previous $2(d+1)$ vertices chosen by the process.  There is one more error term that is also absorbed in the same way. This error term (which is sometimes called `destruction fidelity') comes from the fact that the random variable $ W_{A,j}$ counts vertices while we calculate the expected change by summing over $ (d-1)$-faces. It could be the case that a given vertex $v$ is removed from $ Y_A$ by the elimination of two (or more) different $ (d-1)$-faces. We account for this by noting that the intersection of the sequences of choices for $ \psi(1), \dots \psi(3d+4)$ for two different faces is a lower order term because one of the choices is essentially removed to ensure that we have a sequence in the desired intersection.

Recalling that $Z_{A, j}(\ell) := W_{A, j}((3d + 4)\ell + j) - \frac{1}{3d+4} \left(n \left(1 - p^{|A|} \right) - n^{\alpha} e(t)/2 \right)$ we have that
\begin{eqnarray*}
\mathbb{E}[Z_{A, j}(\ell + 1) - Z_{A, j}(\ell) \mid \mathcal{F}_i] &\geq& |A| Y_A(i)  \binom{d + 1}{2} d! \left( \frac{1}{n^dp} - \frac{(2d + 5) e(t)}{n^{d + 1 - \alpha} p^{d^3/2 + d^2/2 + 1}} \right) \\
&& - n^{1 - d}|A|p^{|A| - 1}d! \binom{d + 1}{2} + n^{\alpha - d} e'(t)/2 - O\left( \frac{1}{n^{2d - 1}}\right) \\
&\geq& \frac{1}{n^d} \left(|A|(np^{|A|} - n^{\alpha} e(t)) \binom{d + 1}{2}d! \left(\frac{1}{p} - \frac{(2d + 5)e(t)}{n^{1 - \alpha}p^{d^3/2 + d^2/2 + 1}}\right)\right) \\
&& - n^{1 - d}|A|p^{|A| - 1}d! \binom{d + 1}{2} + n^{\alpha - d} e'(t)/2 - O\left( \frac{1}{n^{2d - 1}}\right) \\
&\geq& \frac{1}{n^d} \binom{d + 1}{2} d!|A| \left(-\frac{(2d + 5)e(t)p^{|A|}n^{\alpha}}{p^{d^3/2 + d^2/2 + 1}}  - \frac{n^{\alpha} e(t)}{p} + \frac{(2d + 5)(e(t))^2n^{2\alpha}}{np^{d^3/2 + d^2/2 + 1}}\right) \\
&& + n^{\alpha - d} e'(t)/2 - O\left(\frac{1}{n^{2d - 1}}\right). 
\end{eqnarray*}
We choose $e(t)$ so that $e'(t)$ is large enough for this expression to
be positive. So it suffices to have
\[e'(t) > 2 \binom{d + 1}{2} d! \left(d + (d + 2)\binom{d}{2}\right) \frac{(2d + 6) e(t)}{p^{d^3/2 + d^2/2}},\]
using the fact that $|A| \leq d + (d + 2)\binom{d}{2}$. 
So we take
\[e(t) = \exp \left\{ 16d \left(1 - d!\binom{d + 1}{2} t\right)^{-(d^3/2 + d^2/2 - 1)} \right\}.\]
So for this choice of $e(t)$ we have that $e(t) = n^{o(1)}$ for
\[16t\left(1 - d!\binom{d + 1}{2}t\right)^{-(d^3/2 + d^2/2 - 1)} = o(\log n)\]

So putting this together we arrive at a stopping time at least 
\[\left(\frac{1}{d!\binom{d + 1}{2}} - \omega(\log n^{-2/(d^3  + d^2 - 1)})\right).\]

\end{proof}

The last step is to use the submartingale version of Lemma \ref{lem:HA} to go from positive expectation of $Z_{A, j}(\ell + 1) - Z_{A, j}(\ell)$
to all $Z_{A, j}$'s positive with overwhelming probability. To do this we directly apply the following:
\begin{lemma}\cite[Lemma 6]{BohmanTriangleFree}
Suppose $\eta < N/2$ and $a < \eta m$. If $0 = A_0, A_1, ...$ is an $(\eta, N)$-bounded submartingale then
\[\Pr(A_m \leq -a) \leq \exp\left(\frac{-a^2}{3\eta m N}\right).\]
\end{lemma}
\noindent
In a way analogous to the previous case we apply this to the $Z_{A, j}(\ell)$ sequence with $N = \binom{d + 1}{2}$, $\eta = C/n^{d - 1}$ for $C$ a large constant, and $a = -\Omega(n^{3/4})$ coming from $Z_{A, j}(0) = \Omega(n^{3/4})$.

\section{Upper bound for pseudomanifolds}
The upper bound for the pseudomanifold diameter theorem in \cite{CriadoNewman} relied on the fact that the dual graph of a $d$-pseudomanifold is $(d + 1)$-regular. However, here we improve that upper bound by instead using the following fact.  
\begin{theorem}\label{pmConnectivity}
The dual graph of a strongly connected $d$-dimensional pseudomanifold is $(d + 1)$-connected.
\end{theorem}
\noindent
Now, a $(d + 1)$-connected graph on $n$ vertices has diameter at most 
$ (n-2)/(d + 1) + 1$. This simple observation appears as Theorem 1 of \cite{CaccettaSmyth}. The proof is as follows. Let $v$, $u$ in a $(d + 1)$-connected graph $G$ to be two vertices whose distance is $\diam(G)$. Then for each $1 \leq i \leq \diam(G)-1$, let $V_i$ denote the set of vertices at distance $i$ from $v$. Each of $V_1$, ..., $V_{\diam(G)-1}$ is nonempty and has at least $d + 1$ vertices, and more over they are pairwise disjoint subsets of the vertex set of $G \setminus \{u,v\}$. So $(\diam(G)-1)(d + 1) +2 \leq n$.

The fact that the graph of a $d$-pseudomanifold is $(d + 1)$-connected isn't too surprising either, although we couldn't find such a result stated in the literature. It's a natural extension of the simple-polytope case of Balinski's theorem that the graph of a $(d + 1)$-polytope is $(d + 1)$-connected. The closest result to vertex-connectivity of graphs associated to $d$-pseudomanifolds is apparently a result of Barnette \cite{Barnette} that the graph (i.e. the 1-skeleton, not the dual graph) of a $d$-pseuodmanifold is $(d + 1)$-connected. Theorem \ref{pmConnectivity} is then a dual to this result. We give a self-contained topological proof of Theorem \ref{pmConnectivity}.

The key to the proof of Theorem \ref{pmConnectivity} is the  Lemma \ref{pmConnectivityKeyLemma} below. All homology computations are with $\Z/2\Z$ coefficients. For readers unfamiliar with homology we refer to a standard reference such as \cite{Hatcher} and provide a sketch of why the argument works in dimension 2. 

Suppose we have a strongly connected 2-dimensional pseudomanifold $X$ and we have a set $\Sigma$ of at most two triangles so that $X \setminus \Sigma$ is no longer strongly connected. Let $C$ be a strongly connected component of $X \setminus \Sigma$. Then $C$ is a strongly connected pseudomanifold with boundary, i.e. a strongly connected 2-complex with every edge contained in at most 2 triangles. The boundary of $C$ is a graph $H$ in which all vertices have even degree. Moreover $H$ is also a subgraph of $\Sigma$. If $H$ is the boundary of a pseudomanifold with boundary $\Sigma' \subseteq \Sigma$ then adding $\Sigma'$ back in will still not create a path in the dual graph from $C$ to its complement, so $\Sigma \setminus \Sigma'$ is a smaller disconnecting set. On the other hand however if $\Sigma$ is a set of at most two triangle then it is easy to see that any even subgraph of $\Sigma$ is the boundary of a collection of triangles $\Sigma'$.  In the language of homology if $\Sigma$, a set of at most $d$ faces removed from a $d$-pseudomanifold, disconnects the dual graph then it is necessary that $\Sigma$ supports a nontrivial $(d-1)$-dimensional cycle, but the following lemma shows that that is not possible. Here we use ``facet" to refer to a maximal face of a simplicial complex.

\begin{lemma}\label{pmConnectivityKeyLemma}
If $X$ is a simplicial complex of dimension at most $d$ with at most $d$ facets (not necessarily all of the same dimension) then $\tilde{H}_{d - 1}(X) = 0$. 
\end{lemma}
\begin{proof}
By induction. If $d = 1$ then by the assumptions $X$ is either a single vertex or a single edge; in either case it's path connected. Now we suppose that $d \geq 2$ and induct on the number of facets. If $X$ has only 1 facet then $X$ is a simplex and has vanishing homology in all dimensions. Now suppose that $X$ is generated by $\{\sigma_1, \sigma_2, ..., \sigma_{\ell}, \sigma_{\ell + 1}\}$ with $\ell + 1 \leq d$. Letting $X_{\ell}$ denote the complex generated by $\{\sigma_1, ..., \sigma_{\ell}\}$, by the Mayer--Vietoris Sequence we have an exact sequence of homology groups:
\[H_{d - 1}(X_{\ell}) \rightarrow H_{d - 1}(X_{\ell + 1}) \rightarrow \tilde{H}_{d - 2}(X_{\ell} \cap \sigma_{\ell + 1})\]
where $X_{\ell} \cap \sigma_{\ell + 1}$ is the complex obtained by taking the downward closure of $\{\sigma_1 \cap \sigma_{\ell + 1}, \sigma_2 \cap \sigma_{\ell + 1}, ..., \sigma_{\ell} \cap \sigma_{\ell + 1}\}$. By induction on $\ell$, $H_{d - 1}(X_{\ell}) = 0$. Now $X_{\ell} \cap \sigma_{\ell + 1}$ is generated by at most $\ell \leq d - 1$ facets, all of dimension at most $d - 1$, so by induction on $d$, $\tilde{H}_{d - 2}(X_{\ell} \cap \sigma_{\ell + 1}) = 0$. Therefore we have a surjective map from the trivial group onto $H_{d - 1}(X_{\ell + 1})$ so we complete the proof. 
\end{proof}
Note that the lemma is best possible, if we take a single $d$-simplex boundary and have $\sigma_1, ..., \sigma_{d + 1}$ each be a cone of the starting simplex boundary with a unique cone point and base a distinct facet of the central $d$-simplex boundary then the resulting complex has homology in dimension $d - 1$. 
\begin{proof}[Proof of Theorem \ref{pmConnectivity}]
Let $X$ be a $d$-dimensional pseudomanifold and suppose that $\sigma_1, ..., \sigma_{\ell}$ is a minimal collection of at most $d$ facets of $X$ so that $X \setminus \{\sigma_1,  ... , \sigma_{\ell}\}$ is not strongly connected. Let $A$ be the complex generated by $\{\sigma_1, ... , \sigma_{\ell}\}$. By Lemma \ref{pmConnectivityKeyLemma} then $\tilde{H}_{d - 1}(A) = 0$. On the other hand though $\partial_d(X \setminus A) = \partial_d(A)$ where $\partial_d$ refers to the $d$th boundary matrix of $X$ with $\Z/2\Z$ entries, and by $\partial_d(Y)$ we mean $\partial_d \textbf{1}_Y$ for $\textbf{1}_Y$ the indicator vector for the $d$-faces of $Y$.  This follows simply because $X$ being a pseudomanifold implies that $\partial_d(X) = 0$. Now if we let $C$ be a $d$-dimensional strongly connected component of $X \setminus A$ then $\partial_d(C) \subseteq A$ since the strongly connected component partition both $d$-dimensional faces \emph{and} nonmaximal $(d - 1)$-faces (so the boundary matrix has a natural block structure respecting this partitioning). Now $\partial_d(C)$ is a $(d - 1)$-cycle in $A$. Since $\tilde{H}_{d - 1}(A) = 0$, $\partial_d(C)$ is the boundary of some collection $\Sigma$ of the $\{\sigma_1, ..., \sigma_{\ell}\}$. Thus $\Sigma$ is a $d$-pseudomanifold with boundary, but then the dual graph of $(X \setminus A) \cup \Sigma$ cannot be strongly connected since the vertices corresponding the facets in $\Sigma$ are only adjacent to other vertices in $\Sigma$ and vertices in the dual graph of $C$. Thus deleting $\{\sigma_1, ..., \sigma_{\ell}\} \setminus \Sigma$ also disconnects the dual graph of $X$ contradicting minimality. 
\end{proof}
\begin{proof}[Proof of upper bound in Theorem \ref{pmdiametertheorem}]
By Theorem 1 of \cite{CaccettaSmyth}, a $K$-connected graph on $N$ vertices has diameter at most $(n-2)/K+1$. Thus by Theorem \ref{pmConnectivity}, the diameter of a pseudomanifold $X$ is at most $f_d(X)/(d + 1) + 1$. On the other hand, by taking degree sums of $(d - 1)$-faces we have
\[2f_{d - 1}(X) = (d + 1)f_d(X).\]
Trivially $X$ has at most $\binom{n}{d} \leq \frac{n^d}{d!}$ $(d - 1)$-faces so
\[\diam(X) \leq \frac{f_d(X) }{d + 1} +1 = \frac{2 f_{d - 1}(X)}{(d + 1)^2} + 1 \leq \frac{2n^d}{(d + 1)(d + 1)!}\]
for $n$ sufficiently large.
\end{proof}
\section{Closing remarks}
Our main results settle the first order asymptotics of $H_s(n, d)$ and $H_{pm}(n, d)$ for fixed $d$, quantities that were previously studied in \cite{SantosDiamterSurvey, CriadoSantos, CriadoNewman}. But a substantial gap remains in the lower order asymptotics. We could be a bit more precise in the upper bound; for example,  in the simplicial complex case we have
\[H_s(n, d) \leq \frac{1}{d!}\binom{n}{d} \leq \frac{n^d - \Omega(n^{d - 1})}{d!}.\]
There is still a gap between $n^{d - 1}$ 
and $n^d/(\log n)^{1/d^2}$ in the second second order terms from the upper and lower bounds, respectively. While a significant improvement in either bound would be interesting, we do not believe that the lower bound is the correct asymptotic value of $H_s(n, d)$. Our analysis of the process that we introduce here is relatively crude because it is based on bounds 
on the probability that an arbitrary $(d-1)$-face is closed and therefore it does not take advantage of form of the variables we are tracking. So it is likely that the lower bound on the number of steps in the process given by Lemma~\ref{KeyDiffEqLemma} is not correct in the second order term. If one takes the heuristic that the complex of `available' $(d-1)$-faces after $i$ steps of the process is approximately a Linial-Meshulam random complex with each potential $(k-1)$-face appearing with probability $p$ to its furthest limit, then one would not expect the
process to terminate until $ np^d = \Theta( \log n) $. This would then give a lower bound on $ H_s(n,d)$ of the form
\[  \frac{n^d}{d \cdot d!} - O\left( (\log n)^{1/d} n^{d - \frac{1}{d}}   \right).  \]
The proof of such a statement would likely be intricate if using currently available techniques. It seems that one would need to develop sophisticated self-corrected estimates (as in \cite{BohmanFriezeLubetzky, BohmanKeevashr3t, FizGriffithsMorrisr3t}).

One could also generalize what we've done to more general classes of complexes or set systems. In the case of simplicial complexes we map $SC_d(N)$ into the simplex on $n$ vertices subject to the rule that the map is injective on codimension-1 faces. In the case of pseudomanifolds we map $SC_{d}(N)$ into the simplex on $n$ vertices subject to the rule that the map is injective on codimension-2 faces. This naturally extends to trying to construct maps that are injective on codimension-$k$ faces, although it isn't clear from a geometric or topologicial perspective how interesting such a thing would be. 
\begingroup
\bibliography{ResearchBibliography}
\bibliographystyle{amsplain}
\endgroup
\end{document}